\documentclass[11pt,reqno,a4paper]{amsart}

\usepackage[latin1]{inputenc}
\usepackage{tikz}
\usetikzlibrary{shapes,arrows}
\usetikzlibrary{arrows.meta}
\usepackage{forest}
\usepackage{tikz-qtree}

\usetikzlibrary{arrows,shapes,positioning,shadows,trees}

\usepackage{etoolbox}
\usepackage{amsmath}
\usepackage{enumerate}
\usepackage{amssymb}
\usepackage{amscd}
\usepackage{amsthm}
\usepackage{amsfonts}
\usepackage{graphicx}
\usepackage[all,cmtip]{xy}
\usepackage{enumitem}

\patchcmd{\subsection}{-.5em}{.5em}{}{}
\patchcmd{\subsubsection}{-.5em}{.5em}{}{}

\usepackage{enumitem}

\usepackage[T1]{fontenc}
\usepackage[scaled]{helvet} % ss                                                     
\usepackage{courier} % tt                                                            
\usepackage{eulervm}
\normalfont

\makeatother
\usepackage{hyperref}

\numberwithin{equation}{section}

\newcommand{\cA}{\mathcal{A}}

\newcommand{\cH}{\mathcal{H}}

\newcommand{\cR}{\mathcal{R}}

\newcommand{\bC}{\mathbb{C}}

\newcommand{\bE}{\mathbb{E}}

\newcommand{\bR}{\mathbb{R}}

\newcommand{\bZ}{\mathbb{Z}}
\newcommand{\ra}{\rightarrow}

\newcommand{\qor}{\quad \textrm{or} \quad}

\newcommand{\qand}{\quad \textrm{and} \quad}

\def\acts{\curvearrowright}
\newcommand\subsetsim{\mathrel{%
\ooalign{\raise0.2ex\hbox{$\subset$}\cr\hidewidth\raise-0.8ex\hbox{\scalebox{0.9}{$\sim$}}\hidewidth\cr}}}
\newcommand{\eps}{\varepsilon}

\DeclareMathOperator{\supp}{supp}

\DeclareMathOperator{\Dom}{Dom}
\DeclareMathOperator{\shape}{shape}

\DeclareMathOperator{\id}{id}

\theoremstyle{theorem}
\newtheorem{theorem}{Theorem}[section]
\newtheorem{corollary}[theorem]{Corollary}
\newtheorem{proposition}[theorem]{Proposition}
\newtheorem{lemma}[theorem]{Lemma}

\theoremstyle{definition}

\newtheorem{remark}[theorem]{Remark}

\newtheorem*{conjecture}{Conjecture}

\tikzstyle{decision} = [diamond, draw, fill=blue!20, 
    text width=4.5em, text badly centered, node distance=3cm, inner sep=0pt]
\tikzstyle{block} = [rectangle, draw, fill=blue!20, 
    text width=5em, text centered, rounded corners, minimum height=2em]
\tikzstyle{line} = [draw, -latex']
\tikzstyle{cloud} = [draw, ellipse,fill=red!20, node distance=3cm,
    minimum height=2em]

\begin{document}

\title{Bernoulli actions of amenable groups with weakly mixing Maharam extensions}

\author{Michael Bj\"orklund}
\address{Department of Mathematics, Chalmers, Gothenburg, Sweden}
\email{micbjo@chalmers.se}

\author{Zemer Kosloff}
\address{Einstein Institute of Mathematics, Hebrew University of Jerusalem}
\email{zemer.kosloff@mail.huji.ac.il}

\vspace{0.5cm}

\maketitle

\begin{abstract}
We provide a simple criterion for a non-singular and conservative Bernouilli action to have a 
weakly mixing Maharam extension. As an application, we show that every countable amenable group $G$ 
admits a stable type $III_1$ Bernoulli action $G \acts (\{0,1\}^G,\mu)$, answering a 
recent question by Vaes and Wahl. 
\end{abstract}

\section{Introduction and the statements of the main results}

\subsection{The Maharam extension}

Let $G$ be a countable group and let $(X,\mu)$ be a Borel $G$-space, that is to say, a Borel probability measure space $(X,\mu)$ endowed with an action of $G$ by measure-class preserving measurable maps. It is well-known that there exists a $\mu$-conull $G$-invariant 
subset $X' \subset X$ such that the \emph{Radon-Nikodym cocycle}
\[
r_\mu(g,x) = -\log \frac{dg^{-1}\mu}{d\mu}(x),
\]
is defined for all $(g,x) \in G \times X'$. The \emph{Maharam extension} $G \acts (\widetilde{X},\widetilde{\mu})$ associated to 
$(X,\mu)$ is the $G$-action on the set $\widetilde{X} = X' \times \bR$ defined by
\[
g(x,t) = (gx,t+r_\mu(g,x)), \quad \textrm{for $(x,t) \in X' \times \bR$},
\]
which can be readily checked to preserve the \emph{infinite} measure $\widetilde{\mu} = \mu \otimes e^t \, dt$. 

We say that the Maharam extension of $(X,\mu)$ is \emph{weakly mixing} if for every probability measure-preserving ergodic $G$-space $(Y,\nu)$, the diagonal action $G \acts (\widetilde{X} \times Y,\widetilde{\mu} \otimes \nu)$ is ergodic, or equivalently, if the Maharam extension
of $G \acts (X \times Y,\mu \otimes \nu)$ is ergodic. It is not hard to see that if the Maharam extension of $(X,\mu)$ is ergodic, then there is no $G$-invariant $\sigma$-finite measure equivalent to $\mu$. Borel $G$-spaces with weakly mixing Maharam extensions are often
referred to as being of \emph{stable type $III_1$} in the literature. 

\subsection{Non-singular Bernoulli actions}

We note that $G \acts \{0,1\}^G$ by $(gx)_h = x_{g^{-1}h}$ for $g,h \in G$ and $x \in \{0,1\}^G$. Let $(\mu_g)$ be a family of 
probability measures on $\{0,1\}$ indexed by $g \in G$, with the property that there exists $\delta > 0$ such that $\delta \leq \mu_g(0) \leq 1 - \delta$ for all $g \in G$. A classical result of Kakutani \cite[Corollary 1]{K48} tells us that 
the product measure $\mu = \prod_{g \in G} \mu_g$ is non-singular, i.e. its measure-class is preserved by the $G$-action (and  thus $(\{0,1\}^G,\mu)$ is a Borel $G$-space), if and only if
\begin{equation}
\label{kakutani}
\sum_{h \in G} \big(\mu_{gh}(0) - \mu_h(0)\big)^2 < \infty, \quad \textrm{for all $g \in G$}.
\end{equation}
If \eqref{kakutani} holds, we say that $G \acts (\{0,1\}^G,\mu)$ is a \emph{non-singular Bernoulli action}. \\

In this paper we shall primarily be interested in the case when the action is conservative, ergodic and there is no $G$-invariant $\sigma$-finite measure equivalent to $\mu$. Such systems are often said to be of of type $\mathrm{III}$, and the first examples of type $III$-Bernoulli actions for $G = (\bZ,+)$ were constructed by Hamachi in \cite{H81}. As it turns out, ergodic type $III$-actions of any group decompose
further into a one-parameter family of orbit equivalence classes $III_\lambda$, where $0 \leq \lambda \leq 1$, and being of type $III_1$ is
equivalent to having an \emph{ergodic} Maharam extension (see for instance \cite{KW91} for more details). Ulrich Krengel and Benjamin Weiss asked early on: 
\begin{center}
\textit{Which $III_\lambda$-types can occur among non-singular Bernoulli actions?}
\end{center}  
Despite the classical flair of this question, the first examples of non-singular Bernoulli actions of $G = (\bZ,+)$ with \emph{ergodic} Maharam extensions, and thus type $III_1$, were constructed by the second author \cite{Ko11} as late as 2009. A few years later, the second author \cite{Ko14} further showed that for a large class of Bernoulli actions of $G = (\bZ,+)$, being of type $III$ \emph{implies} that the Maharam 
extensions are weakly mixing (in fact, have the $K$-property). In particular, this shows that within this class of examples, only type $III_1$
is possible. Danilenko and Lema\'{n}czyk extended this in \cite{DL18} to a more general class of Bernoulli actions. However, all of these examples are rather special to the additive group of integers as they rely on the notion of a "past" in the group. \\

More recently, Vaes and Wahl in \cite{VW18} studied the question above for general countable groups. In their paper, they formulate
and make substantial progress on, the following conjecture:
\begin{conjecture}\cite{VW18}
	Let $G$ be a countable group with non-vanishing first $\ell^2$-cohomology, or equivalently, suppose that there is a function $f : G \ra \bC$ with $f(0) = 0$, which is not in $\ell^2(G)$, but satisfies
	\begin{equation}
	\label{dirichlet}
	\sum_{h \in G} \big(f(gh)-f(h)\big)^2 < \infty, \quad \textrm{for all $g \in G$}.
	\end{equation}
	Then there is a non-singular Bernoulli action $G \acts (\{0,1\}^G,\mu)$ with an ergodic (or even weakly mixing) Maharam extension. 
\end{conjecture}

If $G$ does not admit a function as above, for instance, if $G$ has property (T), then they show \cite[Theorem 3.1]{VW18} that every non-singular and conservative Bernoulli action of $G$ is equivalent to a probability-measure preserving one, so in particular it cannot have an ergodic Maharam extension. On the other hand, amenable groups do admit plenty of such functions, and in this case they show:

\begin{theorem}\cite[Theorem 6.1]{VW18}
	\label{VWthm}
	Let $G$ be a countable amenable group, and suppose that either:
	\begin{itemize}
		\item $G$ has an infinite order element, or
		\item $G$ admits an infinite subgroup of infinite index. 
	\end{itemize}
	Then there is a non-singular Bernoulli action $G \acts (\{0,1\}^G,\mu)$ with a weakly mixing Maharam extension.
\end{theorem}
They further show that if one is willing to replace the set $\{0,1\}$ with a Cantor set $Z$, then every countable amenable group admits
a non-singular Bernoulli action $G \acts (Z^G,\mu)$ with a weakly mixing Maharam extension. They explicitly ask whether the passage to a Cantor set is necessary, or whether one can keep the base $\{0,1\}$ in this generality. One of the aims of this paper is to prove that one always can. 

\subsection{Main results}

Crucial to the approach of Vaes and Wahl is a two-part criterion \cite[Proposition 6.6]{VW18} for when a non-singular
Bernoulli action of a countable amenable group $G$ has a weakly mixing Maharam extension, which involves a simple 
asymptotic condition on the family $(\mu_g)$ on $\{0,1\}$ and a rather technical assumption that $G$ can be exhausted
by so called "$\lambda$-inessential" subsets. The conditions on $G$ in Theorem \ref{VWthm} are there to guarantee that such
an exhaustion is possible. \\

We show in this paper that the latter technical condition is not needed, and that the asymptotic condition on $(\mu_g)$ stated by Vaes and Wahl, together with the conservativity of the resulting Bernoulli action, is enough to guarantee that the Maharam extension is weakly mixing. 
More precisely, we shall show:

\begin{theorem}
\label{main}
Let $G$ be a countable amenable group and suppose that $G \acts (\{0,1\}^G,\mu)$ is a non-singular conservative Bernoulli action 
with the property that there exist $\delta > 0$ such that $\delta \leq \mu_g(0) \leq 1 - \delta$ for all $g \in G$ and a probability
measure $\lambda$ on $\{0,1\}$ such that
\begin{equation}
\label{asym}
\lim_{g \ra \infty} \mu_g(0) = \lambda(0) \qand \sum_{g \in G} \big(\mu_g(0) - \lambda(0)\big)^2 = \infty.
\end{equation}
Then the Maharam extension of $G \acts (\{0,1\}^G,\mu)$ is weakly mixing.
\end{theorem}

\begin{remark}
Conversely, we prove in the appendix that if a countable group $G$ admits a non-singular Bernoulli action which satisfies the conditions \eqref{asym},
then $G$ must be amenable. 
\end{remark}

Vaes and Wahl \cite[Proposition 4.1 and 6.8]{VW18} prove that every countable amenable group admits a non-singular and conservative Bernoulli action satisfying the conditions \eqref{asym} with $\lambda(0) = \lambda(1) = 1/2$ (this is not quite the way it is phrased in \cite{VW18}; see remark below). In combination with Theorem \ref{main}, this allows us to conclude that:

\begin{corollary}
Every countable amenable group admits a non-singular Bernouilli action $G \acts (\{0,1\}^G,\mu)$ with a weakly mixing Maharam extension. 
\end{corollary}

\begin{remark}
Let $G$ be a countable amenable group, and fix $0 <\delta < 1/2$ and pick a proper function 
$\varphi : G \ra (0,\infty)$ such that $\sum_{g \in G} e^{-c \varphi(g)} = \infty$ for all $c > 0$. Let $F : G \ra [0,\delta)$ be the function from 
\cite[Proposition 6.8]{VW18} constructed from $\varphi$, and set
\[
\mu_g(0) = \mu_g(1) = \frac{1}{2} +  F(g), \quad \textrm{for $g \in G$},
\]
so that $1/2-\delta \leq \mu_g(0) \leq 1/2+\delta$. It follows from \cite[Proposition 6.8]{VW18} that $G \acts (\{0,1\}^G,\mu)$ is a non-singular Bernoulli action with
\[
\lim_{g \ra \infty} \mu_g(0) = \frac{1}{2} \qand \sum_{g \in G} \big(\mu_{g}(0) - \frac{1}{2} \big)^2 = \infty,
\]
and for every $c > 0$
\[
\sum_{g \in G} e^{-c\sum_{h \in G} \big(\mu_{gh}(0) - \mu_h(0)\big)^2} \geq \sum_{g \in G} e^{-c \varphi(g)} = \infty.
\]
By \cite[Proposition 4.1]{VW18}, the last condition implies that $G \acts (\{0,1\}^G,\mu)$ is conservative. 
\end{remark}

\subsection{A few words about the proof of Theorem \ref{main}}

Let $G$ be a countable amenable group and $G \acts (X,\mu) = (\{0,1\}^G,\mu)$ a non-singular and conservative Bernoulli action
satisfying the conditions of Theorem \ref{main}. In particular, if we set
\begin{equation}
\label{defGplus}
G^{+} = \big\{ g \in G \, : \, \mu_g(0) > \lambda(0) \big\}
\qand
G^{-} = \big\{ g \in G \, : \, \mu_g(0) < \lambda(0) \big\},
\end{equation}
then we see that either
\[
\sum_{g \in G^{+}} \big(\mu_g(0) - \lambda(0)\big)^2 = \infty
\qor
\sum_{g \in G^{-}} \big(\mu_g(0) - \lambda(0)\big)^2 = \infty.
\]
Since $\mu_g(0) - \lambda(0) = -(\mu_g(1) - \lambda(1))$, for all $g$, we may without loss of generality assume that the first alternative 
holds (otherwise we just interchange $0$ and $1$). So, from now on, the standing assumptions throughout the rest of the paper are:
\begin{equation}
\label{condsfromnowon}
\lim_{g \ra \infty} \mu_g(0) = \lambda(0) \qand \sum_{g \in G^{+}} \big(\mu_g(0)-\lambda(0)\big)^2 = \infty.
\end{equation}
Throughout the paper, we shall also fix a sequence $(g_m)$ in $G$ with the property that
\begin{equation}
\label{def_gm}
\sum_{m} \big(\mu_{g_m}(0) - \lambda(0)\big)^2 < \infty, 
\end{equation}
and define 
\begin{equation}
\label{def_muprim}
\mu'_{g}(0) 
= \left\{
\begin{array}{cc}
\mu_{g}(0) & \textrm{if $g \neq g_m$ for some $m$}. \\
\lambda(0) & \textrm{if $g = g_m$ for some $m$}.
\end{array}
\right.
\end{equation}
By Kakutani's Theorem mentioned above, the corresponding product measures $\mu'$ and $\mu$ are equivalent. In particular, $(X,\mu')$ is a non-singular
conservative Bernoulli action of $G$, still satisfying the conditions of Theorem \ref{main}, and it is of stable type $III_1$ if and only if $(X,\mu)$ is of 
stable type $III_1$. Hence, from now on, all of our arguments will concern the modified product measure $\mu'$, which we shall denote by $\mu$ to avoid cluttering with primes. The idea of changing the measure along these lines stems from \cite{DL18}.\\

Let now $(Y,\nu)$ be an ergodic probability-measure-preserving $G$-space. By the recent result of Danilenko \cite[Theorem 0.1]{D18}, the diagonal action $G \acts (X \times Y,\mu \otimes \nu)$ is ergodic, and we wish to prove that its Maharam extension is ergodic as well. By \cite[Corollary 5.4]{S77}, this is \emph{equivalent} to showing that for every $\eps > 0$ and $0 \leq t \leq 1$ (or $-1 \leq t \leq 0)$ and $\mu \otimes \nu$-measurable subset $C \subset X \times Y$ with positive measure, 
\begin{equation}
\label{essval}
\mu \otimes \nu
\big( 
C \cap \Big(\bigcup_{g \in G} g^{-1}C \cap \big\{ | r_{\mu \otimes \nu}(g,\cdot) - t| < \eps\big\}\Big)
\big) 
> 0.
\end{equation}
\begin{remark}
A word of clarification might be in order. In Corollary 5.4 in \cite{S77}, to prove ergodicity, one needs to ensure \eqref{essval} for 
\emph{all} $t \in \bR$. However, by \cite[Lemma 3.3]{S77}, the set of $t$ for which \eqref{essval} holds is a closed subgroup of $(\bR,+)$, so in 
particular symmetric, and if contains either of the intervals $[0,1]$ or $[-1,0]$, then it must be equal to $\bR$.
\end{remark}

In our proof, it will be useful to apply some techniques from orbit equivalence theory. We denote by $\cR_G \subset X \times X$ the orbit equivalence
relation, and we denote by $[R_G]$ the collection of all partially defined maps between positive $\mu$-measure subsets of $X$ whose graphs are 
contained in $\cR_G$. It is a standard fact that the Radon-Nikodym derivative can be extended to $[\cR_G]$, and we shall denote it by
$r_{\mu}(V,\cdot)$ for $V \in [\cR_G]$ (well-defined on the domain of $V$). 

Let us now fix a dense subset $\cA$ of the measure algebra of $(X \times Y,\mu \otimes \nu)$. By (the 
first formulation in ) \cite[Lemma 2.1]{CHP87}, to prove \eqref{essval} for a given $t \in \bR$ and for every $\mu \otimes \nu$-measurable subset 
of $X \times Y$ with positive measure, it suffices to show that there exists $M \geq 1$ with the property that for $\eps > 0$ and $C \in \cA$,
we have
\begin{equation}
\label{essval2}
\sup_{V \in [\cR_G]} \mu \otimes \nu
\big( 
C \cap V^{-1}C \cap \big\{ | r_{\mu \otimes \nu}(V,\cdot) - t| < \eps\big\}\Big)
\big) 
\geq
\frac{1}{M}  \, \mu \otimes \nu(C).
\end{equation}
Let $\cA$ denote the subset of the measure algebra of $(X \times Y,\mu \otimes \nu)$ consisting of finite unions of
product sets of the form $A_i \times B_i$, where $A_i \subset X$ are disjoint \emph{cylinder sets} and $B_i \subset Y$ disjoint Borel sets. Then $\cA$ is dense, and to prove \eqref{essval2} for every set in $\cA$, it suffices by disjointness of the sets $A_i \times B_i$, to show that there is a constant $M$ such that \eqref{essval2} holds for every set of the form $C = A_i \times B_i$. 

This is the content of the following proposition which is our main technical result, and whose proof will occupy the rest of the paper.
\begin{proposition}
\label{prop_RS_for_cylinders}
Fix $\eps > 0$, a cylinder set $A \subset X$ and a Borel set $B \subset Y$. 
\begin{itemize}
\item If $\lambda(0) \geq 1/2$, then for every $0 \leq t \leq 1$, there exist two Borel sets $C,D \subset A \times B$ with
positive measures, and a measurable injective map $V : C \ra D$ such that
\begin{equation}
\label{eq_RS_for_cylindersplus}
\mu \otimes \nu(C) \geq \frac{1}{27} \mu \otimes \nu(A \times B)
\qand
\big| r_{\mu \otimes \nu}(V,\cdot) \mid_C - t \big| < \eps.
%\mu \otimes \nu
%\big( 
%(A \times B) \cap \Big(\bigcup_{g \in G} g^{-1}(A \times B) \cap \big\{ | r_{\mu \otimes \nu}(g,\cdot) - t| < \eps\big\}\Big)
%\big) 
%\geq 
%\frac{1}{3} \mu \otimes \nu(A \times B).
\end{equation}
\item If $\lambda(0) < 1/2$, then the same  holds for every $-1 \leq t \leq 0$.
%\begin{equation}
%\label{eq_RS_for_cylindersminus}
%\mu \otimes \nu
%\big( 
%(A \times B) \cap \Big(\bigcup_{g \in G} g^{-1}(A \times B) \cap \big\{ | r_{\mu \otimes \nu}(g,\cdot) - t| < \eps\big\}\Big)
%\big) 
%\geq 
%\frac{1}{3} \mu \otimes \nu(A \times B).
%\end{equation}
\end{itemize}
\end{proposition}
\begin{remark}\label{rmk: CHP}
The previous version of the paper made use of the first criterion for a number to be an essential value as in \cite{CHP87}[Lemma  2.1]. However this is not a sufficient criteria as the example in section 4 in \cite{HS} shows. For this reason the new proof uses the current version of Proposition \ref{prop_RS_for_cylinders}. 
\end{remark}
\subsection{Acknowledgements}
The first author would like to thank the Einstein Institute of Mathematics at Hebrew University in Jerusalem for the hospitality
during the time this paper was written. We would like to thank Stefaan Vaes for pointing out that the error in \cite{CHP87} as mentioned in Remark \ref{rmk: CHP}.

\section{Proof of Proposition \ref{prop_RS_for_cylinders}}
\label{proofProp}

Let $G$ be a countable amenable group and let $(X,\mu) = (\{0,1\}^G,\mu)$ and $\lambda$ be as in Theorem \ref{main}. 
We also fix a probability measure-preserving ergodic $G$-space $(Y,\nu)$. 

\subsection{Notation}

\subsubsection{Shapes of cylinder sets}

If $K \subset G$ is a finite set and $\sigma : K \ra \{0,1\}$, we define the \emph{cylinder set $A$ with $\shape(A) = (K,\sigma)$} by
\[
A = \big\{ x \in X \, : \, x_g = \sigma(g), \enskip \textrm{for all $g \in K$} \big\}.
\]

\subsubsection{Radon-Nikodym cocycles}

By \cite[Lemma 5]{K48}, there exists a $\mu$-conull subset $X' \subset X$ such that the Radon-Nikodym cocycle of 
$(X,\mu)$ satisfies
\begin{equation}
\label{eq_rn}
r_\mu(g,x) = \sum_{h \in G} \big( \log \mu_h(x_h) - \log \mu_{gh}(x_h)\big), \quad \textrm{for all $x \in X'$}.
\end{equation}
Since $\nu$ is $G$-invariant, we have 
\[
r_{\mu \otimes \nu}(g,(x,y)) = r_\mu(g,x), \quad \textrm{for all $(x,y) \in X' \times Y$}.
\]

\subsubsection{The homoclinic relation and its cocycles}

We define the \emph{homoclinic relation} $\cH_\mu$ on $X'$ by
\[
\cH = \big\{ (x,x') \in X' \times X' \, : \, x_g = x'_g, \enskip \textrm{for all but finitely many $g$} \big\}.
\]
For any family $(\eta_g)$ of functions on $\{0,1\}$ we can define the \emph{homoclinic cocycle} (or \emph{Gibbs cocycle}) 
associated to this family by
\[
c_\eta(x,x') = \sum_{g \in G} \big( \eta_g(x_g) - \eta_g(x'_g)\big), \quad \textrm{for $(x,x') \in \cH_\mu$}.
\]
The case when $\eta_g(a) = \log \frac{\mu_g(a)}{\lambda(a)}$ for $a \in \{0,1\}$ will be of special importance to us, and to 
distinguish this case, we omit $\eta$ as an index. In other words, we write
\begin{equation}
\label{defc}
c(x,x') = \sum_{g \in G} \Big( \log \frac{\mu_g(x_g)}{\lambda(x_g)} - \log \frac{\mu_g(x'_g)}{\lambda(x'_g)} \Big), \quad \textrm{for $(x,x') \in \cH_\mu$}.
\end{equation}
We denote by $[\cH_\mu]$ the collection of partially defined one-to-one measurable maps between measurable subsets of $X'$ whose
graphs are subsets of $\cH_\mu$. If $\phi \in [\cH_\mu]$, we write $\Dom(\phi)$ for its domain of definition, and set
\[
\supp(\phi) = \big\{ g \in G \, : \, \phi(x)_g \neq x_g, \enskip \textrm{for some $x \in \Dom(\phi)$}\big\}.
\]
We note that if $A \subset X$ is a cylinder set with $\shape(A) = (K,\sigma)$ and $K \cap \supp(\phi) = \emptyset$, then
\begin{equation}
\label{phiandA}
\phi^{-1}(A) \cap \Dom(\phi) = \big\{ x \in \Dom(\phi) \, : \, \phi(x)_g = \sigma(g), \enskip \textrm{for all $g \in K$} \big\} = A \cap \Dom(\phi),
\end{equation}
and $\mu(A \cap \Dom(\phi)) = \mu(A) \mu(\Dom(\phi))$. 

\begin{remark}\label{rmk: phi and A}
This observation extends beyond cylinder sets in the following way. Given $F\subset G$, denote by $\mathcal{B}_{F}$ the sigma algebra of sets which are measurable with respect to the coordinates $\left\{x_g:\ g\in F \right\}$. Then for any $A\in \mathcal{B}_{G\setminus \supp(\phi)}$ and $B\subset \mathcal{B}_{\Dom(\phi)}$, one has 
\[
\label{phiandA2}
\phi^{-1}\left(A\cap B\right)=A\cap \Phi^{-1}(B), 
\]
and $A$ and $\phi^{-1}(B)$ are independent thus $\mu\left(\phi^{-1}\left(A\cap B\right)\right)=\mu(A)\mu\left(\Phi^{-1}(B)\right)$. Finally, note that every 
$\phi \in [\cH_\mu]$ preserves the measure-class of $\mu$, and thus the Radon-Nikodym derivative $\frac{d\phi^{-1}\mu}{d\mu}$ is well-defined on the set $\Dom(\phi)$.
\end{remark} 

\subsection{Three lemmas}

In what follows, we shall show how Proposition \ref{prop_RS_for_cylinders} can be reduced to three lemmas, whose proofs are postponed to the next sections.   \\

Our first lemma roughly says that the cocycle $c$ defined in \eqref{defc} attains "most values" quite frequently.  

\begin{lemma}
\label{lemma_gibbs1}
Fix $\eps > 0$ and a finite subset $K \subset G$. If either 
\[
\lambda(0) \geq 1/2  \qand t \geq 0
\]
or 
\[
\lambda(0) < 1/2  \qand t \leq 0,
\] 
then there exists $\phi \in [\cH_\mu]$ with finite support such that
\[
K \cap \supp(\phi) = \emptyset \qand \nu(\Dom(\phi)) \geq \frac{1}{3} 
\]
and
\[
c(x,\phi(x))=\log\frac{d \phi^{-1}\mu}{d\mu}(x) \qand \big| c(x,\phi(x)) - t \big| < \eps, \quad \textrm{for all $x \in \Dom(\phi)$}.
\]
\end{lemma}

Our second lemma provides a way to compare the cocycles $r_\mu$ and $c$, and roughly says that if $\phi \in [\cH_\mu]$ and $r_\mu(g,\phi(x))$ is small and $c(x,\phi(x))$ is close to $t$, then 
$r_\mu(g,x)$ is close to $t$ as well. 

\begin{lemma}
\label{lemma_compare_b_and_c}
For every $\eps > 0$ and $\phi \in [\cH_\mu]$ with finite support, there is a finite set $L \subset G$ such that if
\begin{equation}
\label{eq1_compare_b_and_c}
L \cap g(\supp(\phi))  = \emptyset,
\end{equation}
then 
\begin{equation}
\label{eq2_compare_b_and_c}
\big| r_\mu(g,x) - r_\mu(g,\phi(x)) - c(x,\phi(x)) \big| < \eps, \quad \textrm{for all $x \in \Dom(\phi)$}.
\end{equation}
\end{lemma}

Our third lemma is very general and says that if $G \acts (X \times Y,\mu \otimes \nu)$ is conservative, then 
the cocycle $r_{\mu \otimes \nu}(g,\cdot)$ is "frequently" very small for arbitrarily large $g$.

\begin{lemma}
\label{lemma_cons}
For every $\eps > 0$ and finite subset $F \subset G$ and $\mu \otimes \nu$-measurable set $C \subset X \times Y$, we have
\begin{equation}
\label{eq_cons}
\mu \otimes \nu
\big( 
C \cap \Big( \bigcup_{g \notin F} g^{-1}C \cap \big\{ |r_{\mu \otimes \nu}(g,\cdot)| < \eps\big\} \Big)
\big)
= 
\mu \otimes \nu(C).
\end{equation}
\end{lemma}

\subsection{Proof of Proposition \ref{prop_RS_for_cylinders} assuming Lemmas \ref{lemma_gibbs1} and \ref{lemma_compare_b_and_c} and \ref{lemma_cons}}
We start the proof of the following proposition which is an intermediate step.
\begin{proposition}\label{an enhanced CHP lemma}
Fix $\eps > 0$, a cylinder set $A \subset X$ and a Borel set $B \subset Y$. 
\begin{itemize}
	\item If $\lambda(0) \geq 1/2$, then for every $t \geq 0$,  there exists $A'\subset A$  with $\mu\left(A'\right)\geq \frac{1}{3}\mu(A)$ and $\phi \in [\cH_\mu]$ such that   $\phi\left(A'\right)\subset A$   and for all $D\subset A'\times B$,
	\begin{equation}
	\label{eq_RS_for_cylindersplus}
\mu \otimes \nu
\big( 
D \cap \Big(\bigcup_{g \in G} g^{-1}\left(\left(\phi\times id\right)(D)\right) \cap \big\{ | r_{\mu \otimes \nu}(g,\cdot) - t| < \eps\big\}\Big)
\big)=\mu\otimes \nu\left(D\right)
	\end{equation}
	\item If $\lambda(0) < 1/2$, then the same holds with for all $t\leq 0$. 
\end{itemize}	
\end{proposition}
\begin{proof} We only present the proof when $\lambda(0) \geq 1/2$; the case $\lambda(0) < 1/2$ is completely analogous. Let us fix $\eps > 0$, a cylinder set $A \subset X$ with $\shape(A) = (K,\sigma)$ and a Borel set $B \subset Y$. We fix $t \geq 0$ and use Lemma \ref{lemma_gibbs1} to find $\phi \in [\cH_\mu]$ with
\[
K \cap \supp(\phi) = \emptyset \qand \mu(\Dom(\phi)) \geq \frac{1}{3}
\]
such that $|c(x,\phi(x)) - t| < \eps/3$. \\

Set $A' = A \cap \Dom(\phi)$ and note that since $K \cap \supp(\phi) = \phi$, we have
\[
\phi^{-1}(A) \cap \Dom(\phi) = A \cap \Dom(\phi), 
\] 
by \eqref{phiandA}, whence $\phi(A') \subset A$, and $\mu(A') \geq \frac{1}{3}\mu(A)$. In order to shorten notation write $\Psi:=\phi\times id$. \\

By Lemma \ref{lemma_compare_b_and_c}, we can find a finite set $L \supset K$ such that whenever $g \in G$ and $\phi \in [\cH_\mu]$
satisfy 
\[
L \cap g(\supp(\phi)) = \emptyset,
\] 
then
\begin{equation}
\label{rtoc}
\big| r_\mu(g,x) - r_\mu(g,\phi(x)) - c(x,\phi(x)) \big| < \frac{\eps}{3}, \quad \textrm{for all $x \in \Dom(\phi)$}.
\end{equation}
Let $D\subset A'\times B$. We now apply Lemma \ref{lemma_cons} to $ \Psi(D)$ and the finite set 
\[
F = \{g \in G \, : \, L \cap g(\supp(\phi)) \neq \emptyset \big\} \subset G
\]
to conclude that
\[
E = \Psi(D) \cap \Big( \bigcup_{g \notin F} g^{-1}(\Psi (D) \cap \big\{ |r_{\mu \otimes \nu}(g,\cdot)| < \eps/3 \big\}\Big)
\]
satisfies $\mu \otimes \nu(E) = \mu\otimes \nu (\psi (D))$ and thus
\begin{equation}
\label{lowbnd}
\mu \otimes \nu\big(\Psi^{-1}(E)\big) = \mu\otimes \nu(D).
\end{equation}
For every $z = (x,y) \in \Psi^{-1}(E)=(\phi \times \id)^{-1}(E)$, there exists $g_z \notin F$ such
that
\[
|r_{\mu}(g_z,\phi(x))| < \frac{\eps}{3} \qand \left(g_z(\phi(x)),g_z(y)\right)\in \Psi(D).
\]
Set $D_y:=\left\{x\in X:\ (x,y)\in D\right\}\subset A'$ and $\phi_z = g_z \circ \phi \circ g_z^{-1}$ and note that for all $x\in D_y$, 
\[
g_z(x) \in \phi_z^{-1}(\Psi(D)_{g_z(y)}) \qand \supp(\phi_z) = g_z(\supp(\phi)) \qand \Dom(\phi_z) = g_z(\Dom(\phi)).
\]
By Remark \ref{rmk: phi and A}  we see that, 
\[
\phi_z^{-1}\left( \Psi(D)_{g_{z}(y)} \right) \cap \Dom(\phi_z) = \Psi(D)_{g_z(y)} \cap \Dom(\phi_z),
\]
and since $x \in D_y$, and thus $g_z(x) \in \Dom(\phi_z)$, we conclude that $g_z(x) \in \Psi(D)_{g_z(y)}$, in other words $\left(g_z(x),g_z(y)\right)\in (\phi\times id)(D)$. \\

By \eqref{rtoc},
\begin{eqnarray*}
\big| r_{\mu}(g_z,x) - t \big| 
&\leq & 
\big| r_\mu(g_z,x) - r_\mu(g_z,\phi(x)) - c(x,\phi(x))\big| +  \\
&+& \big|r_\mu(g_z,\phi(x))\big| + \big|c(x,\phi(x))-t\big| \\
&<& \frac{\eps}{3} + \frac{\eps}{3} + \frac{\eps}{3} = \eps,
\end{eqnarray*}
so we can conclude that
\[
z = (x,y) \in D \cap g_z^{-1}((\phi\times id)(D)  \cap \big\{ |r_{\mu \otimes \nu}(g_z,\cdot) - t| < \eps\big\},
\]
and thus,
\begin{equation}
\label{incl}
(\phi \times \id)^{-1}(E) \subset  D \cap \Big( \bigcup_{g \notin F} g_z^{-1}((\phi\times id)(D)  \cap \big\{ |r_{\mu \otimes \nu}(g_z,\cdot) - t| < \eps\big\}\Big).
\end{equation}
By combining \eqref{incl} and \eqref{lowbnd}, we get
\[
\mu \otimes \nu \left(D \cap \Big( \bigcup_{g \notin F} g^{-1}((\phi\times id)(D) \cap \big\{ |r_{\mu \otimes \nu}(g,\cdot) - t| < \eps\big\}\Big) \right)=\mu\otimes\nu \left(D\right)
\]
which finishes the proof.
\end{proof}

\begin{proof}[Proof of Proposition \ref{prop_RS_for_cylinders}]
Again we assume that $\lambda(0)\geq 1/2$, as the case $\lambda(0) < 1/2$ is completely analogous. \\

Let $A\subset X$ be a cylinder set and $B\subset Y$ a Borel set and pick $0<t<1$ and $\epsilon>0$ such that $t-\epsilon>0$ and $t+\epsilon<1$. Let us further choose $A'$ and $\phi\in [\cH_\mu]$ as in Proposition \ref{an enhanced CHP lemma}.  

Starting with the set $D_0:=A'\times B$, Proposition \ref{an enhanced CHP lemma} asserts that there are two sets $C_0\subset D_0$ and $E_0\subset (\phi\times id)(D_0)$ and an injective measurable map $V_0:C_0\to E_0$ satisfying for all $z\in C_0$, $(z,V_0(z))\in \mathcal{R}$ and 
\[
\left|\log\left(\frac{dV_o^{-1}\mu}{d\mu}\right)-t\right|<\epsilon.
\]
If $\mu\otimes \nu \left((A'\times B)\setminus{C_0}\right)>0$ and $\mu\otimes \nu \left((\phi(A')\times B)\setminus{E_0}\right)>0$, then we can apply 
Proposition \ref{an enhanced CHP lemma} again, this time to the set $D_1=(A'\times B)\setminus C_0$, and find two positive measure subsets $C_1\subset D_1$ and $E_1\subset (\phi\times id)\left(C_1\right)$ and a partial transformation  $V_1: C_1\to E_1$ satisfying for all $z\in C_1$, $(z,V_1(z)) \in \mathcal{R}$ and the above condition on the Radon-Nikodym derivative. Since $C_0\cap C_1=E_0\cap E_1=\emptyset$ one can without problems define a partial transformation $V:C_0\cup C_1=E_0\cup E_1$. \\

If $\mu\otimes \nu \left((A'\times B)\setminus{C_0\cup C_1}\right)>0$ and  $\mu\otimes \nu \left((\phi(A')\times B)\setminus{D_0\cup D_1}\right)>0$  one can proceed in a similar way and enlarge the domain and the range of $V$. By standard exhaustion type argument one will end up with a possibly terminating sequence of sets $C_0,C_1,...\subset A'\times B$ and $D_0,D_1,...\subset A\times B$ such that either $\cup_{k=0}^\infty C_i=A'\times B$ or $\cup_{k=0}^\infty D_i=\phi(A')\times B$ (modulo $\mu\otimes \nu$ null sets) and a partial transformation $V: \cup_{i=0}^\infty C_i\to \cup_{i=0}^\infty D_i$ such that $x\in \cup_{i=0}^\infty C_i$ $(x,Vx)\in\mathcal{R}$  and 
\[
\left|\log\left(\frac{dV^{-1}\mu}{d\mu}\right)-t\right|<\epsilon.
\]
It remains to note that either $\mu\otimes \nu \left(\cup_{i=0}^\infty C_i\right)=\mu\otimes \nu\left(A'\times B\right)=\frac{1}{3}\mu\otimes \nu(A\times B)$ or $\cup_{i=0}^\infty D_i=\phi(A')\times B$. If the latter case happens then by the condition on the Radon-Nikodym derivatives of $V$ and $\phi$,
\[
e\mu\otimes \nu \left(\cup_{i=0}^\infty C_i\right)\geq e^{t+\epsilon}\mu\otimes \nu \left(\cup_{i=0}^\infty C_i\right)\geq \mu\otimes\nu (\phi(A')\times B),
\]
and thus $\mu\otimes \nu \left(\cup_{i=0}^\infty C_i\right)\geq e^{-1}\mu\otimes\nu(\phi(A)\times B)\geq e^{-2}\mu\otimes\nu\left(A'\times B\right)$, where the
last inequality is established along similar lines as above. Finally, the last term is bigger than $\frac{1}{27}\mu\otimes\nu(A\times B)$.
\end{proof}
\section{Proofs of Lemmas \ref{lemma_gibbs1} and \ref{lemma_compare_b_and_c}}

Throughout this section, we retain the notation introduced in Section \ref{proofProp}. We recall our standing 
assumptions on $(\mu_g)$ and $\lambda$ from Theorem \ref{main} and from \eqref{condsfromnowon}, namely that there exists $\delta > 0$ such
that
\begin{equation}
\label{deltabnd}
\delta \leq \mu_g(0) \leq 1-\delta, \quad \textrm{for all $g \in G$},
\end{equation}
a sequence $(g_m)$ of $G$ such that
\[
\mu_{g_m}(0) = \lambda(0), \quad \textrm{for all $m$},
\]
and 
\begin{equation}
\label{condsagain}
\lim_{g \ra \infty} \mu_g(0) = \lambda(0) \qand \sum_{g \in G^{+}} \big(\mu_g(0)-\lambda(0)\big)^2 = \infty,
\end{equation}
where
\[
G^{+} = \big\{g \in G \, : \, \mu_g(0) > \lambda(0) \big\}.
\]
If $\eta_g : \{0,1\} \ra \bR$ is a sequence of functions, indexed by $g \in G$, we define the associated homoclinic (or Gibbs) cocycle $c_\eta : \cH_\mu \ra \bR$ by
\[
c_\eta(x,x') = \sum_{g \in G} \big(\eta_g(x_g) - \eta_g(x'_g)\big), \quad \textrm{for $(x,x') \in \cH_\mu$}.
\]

\subsection{Proof of Lemma \ref{lemma_gibbs1}}

Lemma \ref{lemma_gibbs1} is an immediate consequence of the following two lemmas, whose proofs are presented below.

\begin{lemma}
\label{prfgibbs}
Let $\eta_g : \{0,1\} \ra \bR$ and suppose that $\lim_{g \ra \infty} \|\eta_g\|_\infty = 0$, and
\begin{equation}
\label{lyap1}
 \sum_{g \in G^{+}} \big(\eta_g(0) - \eta_g(1)\big)^2 = \infty.
\end{equation}
Fix $\eps  > 0$ and a finite subset $K \subset G$. 
\begin{itemize}
\item If 
\begin{equation}
\label{lyap2plus}
\sum_{g \in G^{+}} (\eta_g(0) - \eta_g(1))(\mu_g(0) - \mu_g(1)) = \infty,
\end{equation}
then, for every $t \geq 0$, there exists $\phi_{+} \in [\cH_\mu]$ with
\[
K \cap \supp(\phi_{+}) = \emptyset \qand \nu(\Dom(\phi_{+})) \geq \frac{1}{3} 
\]
such that  $\log\left(\frac{d\phi_{+}^{-1}\mu}{d\mu}\right)|_{\Dom\phi}=c_\eta(\cdot,\phi_{+}(\cdot))$ and 
\[
\big| c_\eta(x,\phi_{+}(x)) - t \big| < \eps, \quad \textrm{for all $x \in \Dom(\phi_{+})$}.
\]
\item If 
\begin{equation}
\label{lyap2minus}
\sum_{g \in G^{+}} (\eta_g(0) - \eta_g(1))(\mu_g(0) - \mu_g(1)) = -\infty,
\end{equation}
then, for every $t \leq 0$, there exists $\phi_{-} \in [\cH_\mu]$ with
\[
K \cap \supp(\phi_{-}) = \emptyset \qand \nu(\Dom(\phi_{-})) \geq \frac{1}{3} 
\]
such that $\log\left(\frac{d\phi_{-}^{-1}\mu}{d\mu}\right)|_{\Dom\phi}=c_\eta(\cdot,\phi_{-}(\cdot))$ and 
\[
\big| c_\eta(x,\phi_{-}(x)) - t \big| < \eps, \quad \textrm{for all $x \in \Dom(\phi_{-})$}.
\]
\end{itemize}
\end{lemma}

\begin{lemma}
\label{etaisgood}
Let $\eta_g = \log \frac{\mu_g}{\lambda}$. Then $\lim_{g \ra \infty} \|\eta_g\|_\infty = 0$, and the sequence $(\eta_g)$ satisfies \eqref{lyap1}. Furthermore,  
\begin{itemize}
\item if $\lambda(0) \geq 1/2$, then $(\eta_g)$  satisfies \eqref{lyap2plus}.
\item if  $\lambda(0) < 1/2$, then $(\eta_g)$ satisfies \eqref{lyap2minus}.
\end{itemize}
\end{lemma}

\subsection{Proof of Lemma \ref{prfgibbs}}

We choose an enumeration $\{h_1,h_2,\ldots\}$ of the set $ G\setminus \{g_1,g_2,\ldots\}$, and fix a finite set $L\subset\mathbb{N}$ such 
that $K\subset \left\{h_k\right\}_{k\in L}\cup \left\{g_k\right\}_{k\in L}$. For $k \in \mathbb{N}$, we define the map $\tau_k : X \ra X$ by
\[
\tau_k(x)_{g} = 
\left\{
\begin{array}{cc}
x_{h_k} & \textrm{if $g=g_k$} \\
x_{g_k} & \textrm{if $g = h_k$}\\
x_{g}  & \textrm{otherwise}
\end{array}
\right.,
\]
which simply permutes the $g_k$ and $h_k$ coordinates. A calculation, using our assumption that $\mu_{g_k}=\lambda$, shows that 
\begin{align}
\label{rntauk}
\log\left(\frac{d\tau_k^{-1}\mu}{d\mu}\right)(x) &=\log \left(\frac{\mu_{h_k}\left(x_{g_k}\right)\mu_{g_k}\left(x_{h_k}\right) }{\mu_{h_k}\left(x_{h_k}\right)\mu_{g_k}\left(x_{g_k}\right)}\right)\\
&= \log \left(\frac{\mu_{h_k}}{\lambda}\right)\left(x_{h_k}\right)-\log \left(\frac{\mu_{h_k}}{\lambda}\right)\left(x_{g_k}\right)\\
&=\eta_{h_k}\left(x_{h_k}\right)-\eta_{h_k}\left(\tau_k(x)_{h_k}\right)=c_\eta\left(x,\tau_k(x)\right).
\end{align}
For $k \in\mathbb{N}$ we write, 
\[
F_k(x):=\eta_{h_k}\left(x_{h_k}\right)-\eta_{h_k}\left(\tau_k(x)_{h_k}\right)
\]
We note that all $(F_k)$ are independent of each other,
\begin{equation}
\label{propF}
\lim_{k \ra \infty} \|F_k\|_\infty = 0,
\end{equation}
and
\begin{align}
\label{meanF}
\int F_kd\mu&=\left(\left(\mu_{h_k}(0)-\mu_{h_k}(1)\right)+\left(\lambda(0)-\lambda(1)\right)\right)\left(\eta_{h_k}(0)-\eta_{h_k}(1)\right) \\
&\geq \big(\eta_{h_k}(0)-\eta_{h_k}(1)\big)(\mu_{h_k}(0) - \mu_{h_k}(1)\big).
\end{align}
and
\begin{equation}
\label{varF}
\int_X F^2_k \, d\mu \geq \frac{1}{4} \, \big(\eta_{h_k}(0)-\eta_{h_k}(1)\big)^2.
\end{equation}
 Let us fix $\eps > 0$ and choose a finite subset $\mathbb{N}\supset M \supset L$ such that $\|F_{h_k}\|_\infty < \eps/2$ for all $k \notin M$. We enumerate
$\mathbb{N}\setminus M = \{k_1,k_2,\ldots\}$ and set 
\[
Z_n = \{k_1,\ldots,k_n\} \qand S_n = \sum_{k \in Z_n} F_{h_k} \qand A_n = \sum_{k \in Z_n} \int_X F_{h_k} \, d\mu.
\]
Since $\delta \leq \mu_g(0) \leq 1-\delta$ for all $g \in G$, it is not hard to see that there is a constant $C_\delta$ such that
\[
\frac{1}{C_\delta} \sum_{k \in Z_n} \big( \eta_{h_k}(0) - \eta_{h_k}(1) \big)^2 
\leq 
\int_X |S_n - A_n|^2 \, d\mu
\leq
C_\delta \, \sum_{k\in Z_n} \big( \eta_{h_k}(0) - \eta_{h_k}(1) \big)^2, 
\]
for all $n$, and thus the variance $B^2_n := \int_X |S_n - A_n|^2 \, d\mu$ tends to infinity as $n \ra \infty$ by \eqref{lyap1}. 
Hence, by Lyapunov's CLT (see e.g. \cite[Theorem 7.1.2]{C01}), 
\[
\lim_n \mu\Big(\Big\{ x \in X \, : \frac{S_n(x)-A_n}{B_n} \geq r \Big\}\Big) = \frac{1}{\sqrt{2\pi}}\int_{r}^\infty e^{-u^2/2} \, du, 
\]
for all $r \in \bR$. In particular, there exist $n_{+}, n_{-} \geq 1$ such that
\[
\mu\big(\big\{ x \in X \, : \, S_n(x) > A_n \big\}\big) = \mu\big(\big\{ x \in X \, : \, \frac{S_n(x)-A_n}{B_n} > 0 \big\}\big) \geq \frac{1}{3}, \quad \textrm{for all $n \geq n_{+}$}
\]
and 
\[
\mu\big(\big\{ x \in X \, : \, S_n(x) < A_n \big\}\big) = \mu\big(\big\{ x \in X \, : \, \frac{S_n(x)-A_n}{B_n} < 0 \big\}\big) \geq \frac{1}{3}, \quad \textrm{for all $n \geq n_{-}$}.
\]
If \eqref{lyap2plus} holds, then $A_n \ra \infty$, so for any $t \geq 0$, the inequality $A_n \geq t$ holds eventually, and thus we can find 
$N_{+} = N_{+}(t) \geq n_{+}$ such that
\[
\mu\big(\big\{ x \in X \, : \, S_{N_{+}}(x) > t \big\}\big) \geq \frac{1}{3}.
\]
Let us from now on fix $t \geq 0$ and an integer $N_{+} = N_{+}(t)$ as above, and set 
\[
E^{+}_t = \big\{ x \in X \, : \, S_{N_{+}}(x) > t \big\}.
\]
We define $T_{+} : E^{+}_t \ra \{1,\ldots,N_{+}\}$ by
\[
T_{+}(x) = \min\big\{ n \geq 1 \, : \, S_n(x) > t \big\},
\]
and $\phi_{+} : E^{+}_t \ra X$ by
\[
\phi_{+}(x)_g = 
\left\{
\begin{array}{cc}
x_{h_k} & \textrm{if there exists $k\in Z_{T_{+}(x)}$ such that $g=g_k $} \\
x_{g_k} & \textrm{if there exists $k \in  Z_{T_{+}(x)}$ such that $g=h_k$}\\
 x_g & \textrm{otherwise}.
\end{array}
\right..
\]
Clearly, 
\[
\Dom(\phi_{+}) = E^{+}_t \qand \supp(\phi_{+}) \subset G_{N_+} \qand (x,\phi_{+}(x)) \in \cH_\mu, \quad 
\textrm{for all $x \in \Dom(\phi_{+})$}.
\]
We note that by \eqref{rntauk}, and the definition of $T_{+}$,
\[
\log\left(\frac{d\phi_{+}^{-1}\mu}{d\mu}\right)(x)=c_\eta(x,\phi_{+}(x)) = \sum_{k \in Z_{T_{+}(x)}} F_k(x) =  S_{T_{+}(x)}(x) > t,
\]
and since $\|F_g\|_\infty < \eps/2$ for all $g \notin M$, we also have
\[
S_{T_{+}(x)}(x) = S_{T_{+}(x)-1}(x) + F_{g_{T_{+}(x)}}(x) < t + \eps,
\]
and thus $|c_\eta(x,\phi_{+}(x)) - r| < \eps$. \\

It remains to show that $\phi_{+} \in [\cH_\mu]$, that is to say, $\phi_{+}$ is one-to-one on $E^{+}_t$. 
Pick two different points $x, x' \in E^{+}_t$. If $T_{+}(x) = T_{+}(x')$, then clearly $\phi_{+}(x)$ and $\phi_{+}(x')$ are distinct, so we may without loss of generality assume that $T_{+}(x) < T_{+}(x')$, whence
\[
S_{T_{+}(x)}(x) > S_{T_{+}(x)}(x'),
\]
Since for all $N\in\mathbb{N}$, $S_{N}\left(\phi_+(x)\right)=-S_N(x)$ we see that 
\[
S_{T_{+}(x)}(\phi_+(x)) < S_{T_{+}(x)}(\phi_+(x'))
\]
which shows that $\phi_{+}(x) \neq \phi_{+}(x')$. \\

If instead \eqref{lyap2minus} holds, then $A_n \ra -\infty$, then the argument above goes through for \emph{negative} $t$ with no essential changes.
%
%\textcolor{red}{Change this!!!!!! If \eqref{lyap2minus} holds, then $A_n \ra -\infty$, so for any $t \leq 0$, the inequality $A_n \leq t$ holds eventually, and thus we can find 
%$N_{-} = N_{-}(t) \geq n_{-}$ such that}
%\[
%\mu\big(\big\{ x \in X \, : \, S_{N_{-}}(x) < t \big\}\big) \geq \frac{1}{3}.
%\]
%Let us from now on fix $t \leq 0$ and an integer $N_{-} = N_{-}(t)$ as above, and set 
%\[
%E^{-}_t = \big\{ x \in X \, : \, S_{N_{-}}(x) < t \big\}.
%\]
%We define $T_{-} : E^{-}_t \ra \{1,\ldots,N_{+}\}$ by
%\[
%T_{-}(x) = \min\big\{ n \geq 1 \, : \, S_n(x) < \frac{t}{2} \big\},
%\]
%and $\phi_{-} : E^{-}_t \ra X$ by
%\[
%\phi_{-}(x)_g = 
%\left\{
%\begin{array}{cc}
%\tau_g(x) & \textrm{if $g \in G_{T_{-}(x)}$} \\
%x & \textrm{if $g \notin G_{T_{-}(x)}$}.
%\end{array}
%\right..
%\]
%Clearly, 
%\[
%\Dom(\phi_{-}) = E^{-}_t \qand \supp(\phi_{-}) \subset G_{N_-} \qand (x,\phi_{-}(x)) \in \cH_\mu, \quad 
%\textrm{for all $x \in \Dom(\phi_{-})$}.
%\]
%The rest of the argument is now completely analogous to previous case.\textcolor{red}{up to here}

\subsection{Proof of Lemma \ref{etaisgood}}

Set $\eta_g = \log \frac{\mu_g}{\lambda}$. We recall the standard inequalities:
\[
\frac{t}{1+t} \leq \log(1+t) \leq t, \quad \textrm{for all $t > -1$}.
\]
In particular, applied to $t = \frac{\mu_g(\cdot)}{\lambda(\cdot)} - 1$, these yield
\[
\frac{\mu_g - \lambda}{\mu_g} \leq \eta_g \leq \frac{\mu_g-\lambda}{\lambda}, \quad \textrm{on $\{0,1\}$},
\]
and thus, by \eqref{deltabnd}, 
\[
\eta_g(0) - \eta_g(1) \geq \Big( \frac{1}{\mu_g(0)} + \frac{1}{\lambda(1)}\Big) \, \big(\mu_g(0) - \lambda(0)\big) 
\geq \Big( \frac{1}{1-\delta} + \frac{1}{\lambda(1)}\Big) \, \big(\mu_g(0) - \lambda(0)\big),
\]
for all $g \in G$. In particular, $\eta_g(0) > \eta_g(1)$ for all $g \in G^{+}$, and by the second condition in \eqref{condsagain}, 
we see that \eqref{lyap1} holds. \\

Set
\[
I = \sum_{g \in G^{+}} \big(\eta_g(0) - \eta_g(1)\big)\big(\mu_g(0) - \mu_g(1)\big) = 2 \, \sum_{g \in G^{+}} \big(\eta_g(0) - \eta_g(1)\big)\big(\mu_g(0) - \frac{1}{2}\big).
\]
If $\lambda(0) \geq 1/2$, then
\[
\mu_g(0) - \frac{1}{2} = \mu_g(0) - \lambda(0) + \lambda(0) - \frac{1}{2} \geq \mu_g(0) - \lambda(0),
\]
and thus, by the second condition in \eqref{condsagain}, 
\[
I \geq 2 \, \Big( \frac{1}{1-\delta} + \frac{1}{\lambda(1)}\Big) \, \sum_{g \in G^{+}} \big(\mu_g(0) - \lambda(0)\big)^2 = \infty.
\]
If $\lambda(0) < 1/2$, then there exists $\eps > 0$ such that the set 
\[
F = \big\{ g \in G^{+} \, : \, \mu_g(0) - \frac{1}{2} \geq -\eps \big\}
\]
is finite. The second condition in \eqref{condsagain}, combined with the fact that $\mu_g(0) - \lambda(0) < 1$ for all $g \in G^{+}$, shows that we have $\sum_{g \in G^{+} \setminus F} \big( \mu_g(0) - \lambda(0) \big) = \infty$, whence
\[
\sum_{g \in G^{+} \setminus F} \big(\eta_g(0) - \eta_g(1)\big)\big(\mu_g(0) - \frac{1}{2}\big) < - \eps \, \Big( \frac{1}{1-\delta} + \frac{1}{\lambda(1)}\Big) \,\sum_{g \in G^{+} \setminus F}  \big( \mu_g(0) - \lambda(0) \big) = -\infty.
\]
We conclude that $I = -\infty$.

\subsection{Proof of Lemma \ref{lemma_compare_b_and_c}}

We recall the definitions of the cocycles $r_\mu$ and $c$ from Section \ref{proofProp}. In particular, by \eqref{eq_rn}, we know that for all $x \in X'$, 
\[
r_\mu(g,x) = \sum_{h \in G} \big( \log \mu_h(x_h) - \log \mu_{gh}(x_h)\big), \quad \textrm{for all $x \in X'$},
\]
and for all $(x,x') \in \cH_\mu$,
\[
c(x,x') = \sum_{g \in G} \Big( \log \frac{\mu_g(x_g)}{\lambda(x_g)} - \log \frac{\mu_g(x'_g)}{\lambda(x'_g)} \Big).
\]
Fix $\eps > 0$ and $\phi \in [\cH_\nu]$ with finite support and set
\[
L = \Big\{ g \in G  \, : \, \max_{a \in \{0,1\}}\big| \log \frac{\mu_g(a)}{\lambda(a)} \big| \geq \frac{\eps}{2|\supp(\phi)|} \Big\}.
\]
Since $\delta \leq \mu_g(0) \leq 1-\delta$ for all $g \in G$ and $\mu_g(0) \ra \lambda(0)$ as $g \ra \infty$, the set $L$ is finite. We note that for every $g \in G$ and $x \in \Dom(\phi)$,
\begin{eqnarray*}
r_\mu(g,x) - r_\mu(g,\phi(x)) 
&=& 
\sum_{h \in \supp(\phi)} 
\Big( 
\log \frac{\mu_h(x_h)}{\lambda(x_h)} - \log \frac{\mu_h(\phi(x)_h)}{\lambda(\phi(x)_h)}\big) 
\Big) + \\
&+& 
\sum_{h \in \supp(\phi)} \log \frac{\lambda(x_h) \cdot \mu_{gh}(\phi(x)_h)}{\mu_{gh}(x_h) \cdot \lambda(\phi(x)_h)} \\
&=&
c(x,\phi(x)) + \sum_{h \in \supp(\phi)} \log \frac{\lambda(x_h) \cdot \mu_{gh}(\phi(x)_h)}{\mu_{gh}(x_h) \cdot \lambda(\phi(x)_h)}.
\end{eqnarray*}
If $L \cap g(\supp(\phi)) = \emptyset$, then the absolute value of the last sum is bounded by $\eps$, which finishes
the proof.

\section{Proof of Lemma \ref{lemma_cons}}

We retain the notation from Section \ref{proofProp}. It is not hard to see that $G \acts (X,\mu)$ essentially free. Since we have assumed that it is also conservative, so is the diagonal action $G \acts (X \times Y,\mu \otimes \nu)$. Hence Lemma \ref{lemma_cons} follows from the following general result, which is surely known to experts, but we include it for completeness.

\begin{lemma}
If $G \acts (Z,\xi)$ is conservative and essentially free, then, for every $\eps > 0$, Borel set $C \subset Z$ and \emph{finite} subset $F \subset G$, 
\begin{equation}
\label{eq_conservativity}
\xi\big(C \cap \Big( \bigcup_{g \notin F} g^{-1}C \cap \big\{ |r_\xi(g,\cdot)| < \eps \big\}\Big) \big) = \xi(C).
\end{equation}
\end{lemma}

\begin{proof}
Since $G \acts (Z,\xi)$ is conservative, it follows from \cite[Theorem 4.2 and Theorem 5.5]{S77} that its Maharam extension
$G \acts (\widetilde{Z},\widetilde{\xi})$ is conservative as well. Since $G \acts (Z,\xi)$ is essentially free, so is $G \acts (\widetilde{Z},\widetilde{\xi})$. Hence by Halmos recurrence theorem \cite[Proposition 1.6.2]{A97}, for every measurable set $\widetilde{C} \subset \widetilde{Z}$, we have $\widetilde{D} = \widetilde{C}$ modulo $\widetilde{\xi}$-null sets, where
\begin{equation}
\label{div1}
\widetilde{D} = \big\{ (z,t) \in \widetilde{C} \, : \, \sum_{g \in G} \chi_{\widetilde{C}}(gz,t+r_\xi(g,z)) = \infty \big\}.
\end{equation}
Let $\delta > 0$ and fix a Borel set $C \subset Z$. If we define $\widetilde{C} = C \times (-\delta,\delta)$ and let $\widetilde{D}$ denote
the corresponding set defined in \eqref{div1}, then Fubini's Theorem tells us that there is a $\xi$-conull subset $C' \subset C$ such
that 
\[
\widetilde{D}_z := \big\{t \in \bR \, : \, (z,t) \in \widetilde{D} \big\} = (-\delta,\delta), \quad \textrm{modulo $\xi$-null sets},
\]
for all $z \in C'$. In particular, for all $z \in C'$, we have
\[
\int_{\widetilde{D}_z}  \sum_{g \in G} \chi_{\widetilde{C}}(gz,t+r_\xi(g,z)) \, e^{t} \, dt = \sum_{g \in G} \chi_C(gz) \rho_\delta(r_\xi(g,z)) = \infty,
\]
where
\[
\rho_\delta(s) = \int_{-\delta}^\delta \chi_{(-\delta,\delta)}(s+t) e^{t} \, dt, \quad \textrm{for $s \in \bR$}.
\]
Let us now fix $\eps > 0$ and choose $\delta > 0$ so small so that $\rho_\delta \leq 2 \chi_{(-\eps,\eps)}$. Then, with the notations 
above,
\[
\sum_{g \in G} \chi_C(gz) \chi_{(-\eps,\eps)}(r_\xi(g,z)) = \infty, \quad \textrm{for all $z \in C'$},
\]
and thus, for every finite subset $F \subset G$,
\[
\sum_{g \notin F} \chi_C(gz) \chi_{(-\eps,\eps)}(r_\xi(g,z)) = \infty, \quad \textrm{for all $z \in C'$},
\]
We conclude that for every finite set $F \subset G$,
\[
C \cap \Big( \bigcup_{g \notin F} g^{-1}C \cap \big\{ |r_\xi(g,\cdot)| < \eps \big\} \Big) \supset \big\{ z \in C \, : \, \sum_{g \notin F} \chi_C(gz) \chi_{(-\eps,\eps)}(r_\xi(g,z)) = \infty\big\} \supset C'.
\]
Since $\xi(C') = \xi(C)$, we are done.
\end{proof}

\appendix

\section{On the role of amenability}

\begin{proposition}
\label{nonamen_nonexist}
Let $G$ be a finitely generated \emph{non-amenable} group and suppose that $G \acts (\{0,1\}^G,\mu)$ is a non-singular Bernoulli action with the property that there exist $\delta > 0$ such that $\delta \leq \mu_g(0) \leq 1 - \delta$ for all $g \in G$ and a probability measure $\lambda$ on $\{0,1\}$ such that 
\[
\lim_{g \ra \infty} \mu_g(0) = \lambda(0).
\]
Then $\mu$ is equivalent to the $G$-invariant probability measure $\prod_{g} \lambda$. In particular, it is not of type $III$.
\end{proposition}

\begin{remark}
The proof below can be adapted to also deal with infinitely generated groups, but providing references in this generality 
would become harder. 
\end{remark}

Since $G \acts (\{0,1\}^G,\mu)$ is non-singular and $\delta \leq \mu_g(0) \leq 1 - \delta$ for all $g \in G$, Kakutani's criterion \ref{kakutani} shows that 
\[
\sum_{h \in G} \big(\mu_{gh}(0) - \mu_h(0)\big)^2 < \infty, \quad \textrm{for all $g \in G$}.
\]
Let $S$ be a finite and symmetric generating set of $G$. Since $G$ is non-amenable, we can, by \cite[Corollary 3]{BV97}, write 
\[
\mu_g(0) = u(g) + v(g), \quad \textrm{for all $g \in G$},
\]
where $v \in \ell^2(G)$ and 
\begin{equation}
\label{harm}
\frac{1}{|S|} \sum_{s \in S} u(gh) = u(g), \quad \textrm{for all $g \in G$}.
\end{equation}
Since $v \in \ell^2(G)$ and $G$ is infinite, we must have $\lim_{g} v(g) = 0$, whence $\lim_{g} u(g) = \lambda(0)$. There are many ways to see why $u$ has to be a constant (and thus equal to $\lambda(0)$). For instance, let $(z_n)$ be the simple random walk on $G$ with
steps in $S$. Then, $z_n \ra \infty$ almost surely (since $G$ is non-amenable), and thus $u(gz_n) \ra \lambda(0)$ for all $g \in G$. On the other hand, by \eqref{harm} and the dominated convergence theorem, 
\[
u(g) = \lim_{n} \bE[u(gz_n)] = \lambda(0), \quad \textrm{for all $g \in G$}.
\] 
We conclude that $\mu_g(0)-\lambda(0) \in \ell^2(G)$, and thus $\mu$ and $\prod_{g} \lambda$ are equivalent by \cite[Corollary 1]{K48}.

%\begin{lemma}{\cite[Corollary 1]{K48}}
%\begin{equation}
%\label{kakutani_eq}
%\sum_{w \in W} \big(\mu_w(0) - \nu_w(0)\big)^2 < \infty 
%\end{equation}
%\begin{equation}
%\label{kakutani_sing}
%\sum_{w \in W} \big(\mu_w(0) - \nu_w(0)\big)^2 = \infty 
%\end{equation}
%\end{lemma}
%
%\begin{corollary}
%\begin{equation}
%\label{kakutani_quasi}
%\sum_{w \in W} \big(\mu_{g^{-1}w}(0) - \mu_w(0)\big)^2 < \infty
%\end{equation}
%\end{corollary}

\end{document}